\documentclass[11pt]{article}

\usepackage{amsmath}
\usepackage{amssymb}

 \newcommand{\beq}{\begin{equation}}
\newcommand{\eeq}{\end{equation}}

\numberwithin{equation}{section}

\newtheorem{theorem+}           {Theorem}      [section]
\newtheorem{definition+}  [theorem+]  {Definition}
\newtheorem{lemma+}  [theorem+]  {Lemma}
\newtheorem{corollary+}  [theorem+]  {Corollary}
\newtheorem{proposition+}  [theorem+]  {Proposition}
\newtheorem{example+}  [theorem+]  {Example}

\newenvironment{theorem}{\begin{theorem+}\sl}{\end{theorem+}\rm}
\newenvironment{definition}{\begin{definition+}\rm}{\end{definition+}\rm}
\newenvironment{lemma}{\begin{lemma+}\sl}{\end{lemma+}\rm}
\newenvironment{corollary}{\begin{corollary+}\sl}{\end{corollary+}\rm}
\newenvironment{proposition}{\begin{proposition+}\sl}{\end{proposition+}\rm}
\newenvironment{example}{\begin{example+}\rm}{\end{example+}\rm}
\newenvironment{proof}{\medbreak\noindent{\it Proof.}\rm}{\hfill$\square$\rm}

\newcommand{\vph}{\varphi}

\renewcommand{\Bbb}{\mathbb}
\newcommand{\Z}{{\Bbb  Z}}

\newcommand{\R}{{ \Bbb R}}

\newcommand{\Rnm}{{ \Bbb R}_-^n}
\newcommand{\Rnk}{{ \Bbb R}_-^k}
\newcommand{\Rnp}{{ \Bbb R}_+^n}
\newcommand{\Rnpt}{{ \Bbb R}_+^2}
\newcommand{\C}{{\Bbb  C}}
\newcommand{\Cn}{{\Bbb  C\sp n}}

\newcommand{\D}{{\Bbb D}}

\newcommand{\fra}{{\mathfrak a}}
\newcommand{\frad}{{\mathfrak a}_\bullet}

\newcommand{\F}{{\mathcal F}}

\newcommand{\cO}{{\mathcal O}}
\newcommand{\cH}{{\mathcal H}}
\newcommand{\cJ}{{\mathcal J}}
\newcommand{\cI}{{\mathcal I}}
\newcommand{\cR}{{\mathcal R}}
\newcommand{\cA}{{\mathcal A}_0}

\newcommand{\Df}{{\mathfrak D}}
\newcommand{\frj}{{\mathfrak j}}

\newcommand{\PSH}{{\operatorname{PSH}}}

\newcommand{\MW}{{\operatorname{MW_0}}}
\newcommand{\MWO}{{\operatorname{MW_0}}}

\newcommand{\AW}{{\operatorname{AW_0}}}
\newcommand{\AAW}{{\operatorname{AAW_0}}}
\newcommand{\UAW}{{\operatorname{IAW_0}}}
\newcommand{\OAW}{{\operatorname{SAW_0}}}

\newcommand{\Wx}{{\operatorname{W_0}}}

\begin{document}

\baselineskip=17pt

\begin{center}
{\Large\bf Analytic approximations of plurisubharmonic singularities}
\end{center}

\begin{center}
{\large Alexander Rashkovskii}
\end{center}

\begin{center}
{\it Tek/Nat, University of Stavanger, 4036 Stavanger, Norway\\
 E-mail: alexander.rashkovskii@uis.no}
\end{center}

\vskip1cm

\begin{abstract}\noindent
We study several classes of isolated singularities of plurisubharmonic functions that can be approximated by analytic singularities with control over their residual Monge--Amp\`ere masses.  They are characterized in terms of Green functions for Demailly's approximations, relative types, and valuations. Furthermore, the classes are shown to appear when studying graded families of ideals of analytic functions and the corresponding asymptotic multiplier ideals.
\medskip


{\sl Key words}: plurisubharmonic singularity, analytic singularity, Demailly's approximation, Monge--Amp\`ere mass, multiplier ideal, asymptotic multiplier ideal
\end{abstract}

\section{Introduction}
Let $u$ be a plurisubharmonic ({\it psh}) function on a complex manifold $X$. We are interested in the asymptotic behavior of $u$ near a point $x$ with $u(x)=-\infty$. The {\it singularity} of $u$ at $x$ is the class of psh functions $v$ satisfying $v=u+O(1)$ near $x$. Our considerations are mostly local, so we assume $X$ to be a domain of $\Cn$ and $x=0$.

Already in dimension $1$, the asymptotic behavior of a psh singularity can be quite complicated. However, everything becomes easy when one considers functions that are harmonic in a punctured neighborhood of $0$. Namely, in this case,  $u(z)=\nu\log|z|+O(1)$ as $z\to 0$, where $\nu$ is the residual Riesz mass of $u$ at the origin.

In several variables, this corresponds to psh functions that are {\it maximal} outside $0$, which means, if $0$ is an isolated singularity point, that they satisfy the complex Monge--Amp\`ere equation $(dd^cu)^n=\tau\delta_0$, $\tau> 0$. Any isolated singularity has a 'maximization' that keeps its standard characteristics such as Lelong number, residual Monge--Amp\`ere mass, integrability index, etc. In this sense, we lose little by restricting ourselves to maximal psh functions.

Even in this class, the variety of singularities is enormous. For instance, their collection contains the functions   $u=a\log|F|+O(1)$ generated by holomorphic mappings $F$. Such {\it analytic singularities} are comparatively well studied, and we consider them as "simple" objects to be used for investigation of arbitrary maximal psh singularities.
In  \cite{FaJ}, \cite{BFaJ}, this idea was realized by extending the notion of a valuation from analytic to psh functions, and a number of results have been obtained there. In particular, such a dual object as formal psh functions on the space of valuations has been introduced and shown to carry most of the information on the singularity. In \cite{FaJ2}, this was used for computing the complex singularity exponent and multiplier ideals for any psh function in ${\Bbb C}^2$.

While the valuative approach is based significantly on technique of algebraic geometry, we would like here to exploit rather a direct approximation of psh singularities by analytic ones.
There are certain indications that this could work. First, we have the classical Lelong--Bremermann theorem on uniform approximation of continuous psh functions by functions  $\max_{i}a_i\log|f_i|$.
Second, a celebrated theorem due to Demailly states that any psh function $u$ can be approximated by a sequence of functions $\Df_ku$ with analytic singularities, converging to $u$ pointwise and in $L_{loc}^1$, and these {\it Demailly approximants} keep track on the singularity of $u$: for example, their Lelong numbers converge to the Lelong number of $u$; see Section~\ref{ssecDem} for details. Actually, the valuative approach makes systematical use of Demailly's approximation theorem as a bridge between psh and analytic functions.

On the other hand, assuming $0$ to be an isolated singularity point of $u$ (such functions will be called {\it weights}), it is not clear whether the residual Monge--Amp\`ere masses $(dd^c\Df_ku)^n(\{0\})$ of the Demailly approximants converge to that of $u$, even if $u$ is maximal outside $0$. The residual mass is an important characteristic of a psh singularity (for instance, $(dd^c\log|F|)^n(\{0\})$ is the multiplicity of the mapping $F:\: X\to\Cn$ at $0$), so this uncertainty is rather regretful. For example, since the functions $\Df_ku$ for any $u$ with zero Lelong number are locally bounded and hence of zero residual mass, the convergence of the masses would give a negative answer to the following well-known

\medskip\noindent
{\bf Zero Lelong Number Problem}\footnote{To the best of our knowledge, the question was first asked in a preprint version (1999) of \cite{FaG} and, independently, in the author's lectures at Summer Semester in Complex Analysis at Feza Gursey Institute, Istanbul (July 1999), see \cite{R-surv}.}:  {\sl Does there exist a psh function $u$ with zero Lelong number at $0$ and positive residual mass $(dd^cu)^n(\{0\})$?}

\medskip

In the present note, we study a few classes of psh singularities arising from different kinds of approximation by analytic ones, and we focus mainly on the question: {\sl How does analytic approximability of psh singularities affect their properties?}

\medskip

One of our tools is the notion of generalized {\it Green function} $G_\psi=G_{\psi,D}$ as the upper envelope of negative psh functions on a small neighborhood $D$ of $0$, with the singularity $\psi$ (see Section~\ref{ssec:Greenf} for details); we use it as a way of uniformization for families of the singularities. The choice of the domain $D$ is unimportant, it needs only to be bounded and hyperconvex.

 In spite of the convergence of the Demailly approximants $\Df_k\psi$ to $\psi$, there is no immediate reason for their Green functions to converge to $G_\psi$ (for instance, the Green functions of $\psi_k=\max\{\psi,-k\}$ equal zero for any $\psi$). It turns out that the convergence depends on the way $\psi$ can be approximated by analytic singularities.

Our first main result concerns psh weights $\psi$ that can be approximated by analytic weights $\psi_\epsilon$ in the sense that
$(1+\epsilon)\psi_\epsilon +O(1)\le  \psi\le (1-\epsilon)\psi_\epsilon +O(1)$ near $0$, see Definition~\ref{def:aas}.
We call them {\it asymptotically analytic} singularities. These include, for example, exponentially H\"{o}lder continuous weights and, more generally, {\it tame} weights as introduced in \cite{BFaJ}. In fact, we do not now if there exist maximal weights that are not asymptotically analytic.

\medskip\noindent
{\bf Theorem A.} {\sl Let $\psi$ be a maximal weight. Then the following conditions are equivalent:
\begin{enumerate} \item[(i)] $\psi$ has an asymptotically analytic singularity;
\item[(ii)] the sequence $G_{\Df_k\psi}/G_{\psi}$ converges to $1$, uniformly on $D\setminus\{0\}$, as $k\to\infty$;
\item[(iii)] for any psh function $u$ near $0$, we have $\sigma(\Df_ku,\psi)\to\sigma(u,\psi)$ as $k\to\infty$.
\end{enumerate}}
\medskip

Here  $\sigma(u,\psi)=\liminf_{z\to 0} {u(z)}/{\psi(z)}$ is the {\it relative type} of $u$ with respect to $\psi$ \cite{R7}. Since $\sigma(u,\log|z|)$ is the Lelong number of $u$, statement (iii) contains Demailly's result on convergence of the Lelong numbers of the approximants $\Df_ku$. Condition (ii) implies that, for any $N>0$, the functions $\max\{G_{\Df_k\psi},-N\}$ converge, as $k\to\infty$, to $\max\{G_{\psi},-N\}$, uniformly on $D$, which guarantees the convergence of the residual masses $(dd^c\Df_k\psi)^n(\{0\})$ to that of $\psi$, as well as continuity of $G_\psi$.

On the other hand, the uniform convergence in (ii) indicates that asymptotical analyticity may be an unnecessarily  restrictive condition for the convergence of the residual masses. The largest class of maximal weights $\psi$ with controlled residual masses of $\Df_k\psi$ is described as follows.

\medskip\noindent
{\bf Theorem B.} {\sl Let $\psi$ be a maximal weight. Then the following conditions are equivalent:
\begin{enumerate} \item[(i)] $\psi$ has an {\rm inf-analytic singularity}, i.e., can be represented as the limit of a decreasing sequence of {maximal} analytic weights;
\item[(ii)] $G_\psi=\inf_kG_{\Df_k\psi}$;
\item[(iii)] $(dd^c\psi)^n(\{0\})=\inf_k(dd^c\Df_k\psi)^n(\{0\})$;
\item[(iv)] there exist a countable set of divisorial valuations $\cR_i$ and positive numbers $s_i$ such that $\sigma(u,\psi)=\inf_i s_i\cR_i(u)$ for any function $u$ plurisubharmonic near $0$.
\end{enumerate}
In addition, for any inf-analytic singularity $\psi$, we have $(dd^c{\Df_{k}\,\psi})^n(\{0\})\to (dd^c\psi)^n(\{0\})$ and $G_{\Df_k\psi}\to G_\psi$ in $L^n(D)$.}
\medskip

In {(iv)}, the valuations $\cR_i$ are understood as the aforementioned extension of valuations from the ring $\cO_0$ to psh singularities \cite{BFaJ}. When both $u$ and $\psi$ have analytic singularities, the representation {(iv)} was proved, by algebraic methods, in \cite{LT}; note that in commutative algebra, the object corresponding to the relative type is known as the {Samuel asymptotic function}.
For arbitrary $u$ and tame $\psi$, the representation was established, by valuative methods, in \cite{BFaJ}. Note that asymptotical/inf-analyticity can be described in terms of formal psh functions on valuations as well. Namely, a maximal weight $\psi$ is inf-analytic if and only if it is generated, in the sense of \cite{BFaJ},  by a formal psh function $\hat\psi$ on the space $\mathcal V$ of valuations on $\cO_0$, and $\psi$ is asymptotically analytic if and only if $\hat\psi$ is continuous on $\mathcal V$  (see a remark after Corollary~\ref{cor:hA1}). We do not explore these issues here because it would require considerable involvement of the machinery of algebraic geometry, which does not fit well into the present exposition.

Nevertheless, we consider here certain algebraic structures because they turn out to invoke a new type of analytic approximability of psh singularities. Namely, this one arises from consideration of asymptotic multiplier ideals; the notion appeared in \cite{ELS1} and since then it has been of great interest in algebraic geometry.

We show first that, given a graded family $\fra_\bullet$ of primary ideals $\fra_k\subset\cO_0$, the Green functions $G_{\fra_k}$ for the singularities determined by the generators of $\fra_k$ converge, after the rescaling $G_{\fra_k}\mapsto k^{-1}G_{\fra_k}$, to a maximal weight $G_{\fra_\bullet}$, while the Green functions $G_{\frj_k}$ of the corresponding asymptotic multiplier ideals $\frj_k$ converge (also after the rescaling) to a maximal weight $G_{\frj_\bullet}\ge G_{\fra_\bullet}$. The multiplicities $e({\fra_\bullet})$ and $e({\frj_\bullet})$, in the sense of \cite{Mustata}, are just the Monge--Amp\`ere masses of $G_{\fra_\bullet}$ and $G_{\frj_\bullet}$, respectively, and the equality $e({\fra_\bullet})=e({\frj_\bullet})$ is equivalent to $G_{\fra_\bullet}=G_{\frj_\bullet}$. The latter holds true if $G_{\fra_\bullet}$ is inf-analytic, which implies, in particular, the result $e({\fra_\bullet})= e({\frj_\bullet})$ in the case of monomial ideals $\fra_k$, obtained in \cite{Mustata}.

For the ideals $\fra_k$ defined by the conditions $ \sigma(\log|f|,\psi)\ge k$ for a maximal weight $\psi$, we get

\medskip\noindent
{\bf Theorem C.} {\sl Let $\psi$ be a maximal weight. Then
\begin{enumerate} \item[(i)] $G_{\fra_\bullet}=G_\psi$ if and only if the weight $\psi$ has a {\rm sup-analytic singularity}, i.e., can be approximated {\sl from below} by maximal analytic singularities;
\item[(ii)] $G_{\frj_\bullet}=G_{\fra_\bullet}=G_\psi$ if and only if $\psi$ is both inf-analytic and sup-analytic;
\item[(iii)] for $\psi$ asymptotically analytic, the ratios $G_{\fra_k}/G_{\frj_k}$ converge to $1$, uniformly on $D\setminus\{0\}$, as $k\to\infty$.
\end{enumerate}}
\medskip

 Previously, the equality $e({\fra_\bullet})= e({\frj_\bullet})$ was known for the ideals $\fra_k=\{f:v(f)\ge k\}$ determined by quasimonomial (Abhyankar) valuations $v$  \cite{ELS2}, \cite{Mustata}. This is covered by Theorem~C(iii) because, as was shown in \cite{BFaJ}, for every such valuation $v$ there exists a maximal tame weight $\psi$ such that $v(f)=\sigma(\log|f|,\psi)$ for all $f\in\cO_0$. Such 'quasimonomial' weights $\psi$ form, however, just a small part of maximal tame weights.

  A new feature of our approach is that we show what the "limits" of the families $\fra_k$ and $\frj_k$ -- or, more precisely, the scaled limits of their logarithmic images $\log|\fra_k|$ and $\log|\frj_k|$ -- are: they are the collections of psh functions $u$ satisfying $\sigma(u,G_{\fra_\bullet})\ge 1$ and $\sigma(u,G_{\frj_\bullet})\ge 1$, respectively. Note that a valuative approach to asymptotic multiplier ideals was recently developed in \cite{JM}, \cite{JM1}.

\medskip

The paper is organized as follows. Section~2 contains a background on psh singularities. In Section~3 we introduce asymptotically analytic singularities. Convergence of the Green functions of Demailly's approximants is considered in Section~4. Specifically, the equivalence of (i) and (ii) in Theorem~A is proved in Section~4.2, and Theorem~B, except for its assertion (iv), is proved in Section~4.3. Asymptotic multiplier ideals are treated in Section~5 where Theorem~C is proved. Finally, in Section~6 we study relative types of Demailly's approximants and finish the proofs of Theorems~A and B; there we also represent inf-analytic weights as lower envelopes of certain analytic disk functionals.


\section{Preliminaries}

For basics on plurisubharmonic (psh) functions and the complex Monge--Amp\`ere operator, we refer the reader to \cite{Kl}. In this section, we recall some notions of particular importance for us and set the corresponding notation.

\subsection {Plurisubharmonic singularities}

Let $X$ be a complex manifold of dimension $n$, which we will think of as a domain in $\Cn$ containing the origin, and let $\PSH_0$ be the collection of germs of
psh functions at $0$. We will say that a psh function $u$ is {\it
singular} at $0$ if $u$ is not bounded (below) in any neighborhood of
the origin.

The equivalence class ${\rm cl}(u)$ of $u\in\PSH_0$ with respect to the relation "$u\sim v$ if $u(z)=v(z)+O(1)$" will be called the {\it singularity} of $u$; in \cite{Za0}, a closely related object was introduced under the name {"standard singularity"}.

Germs with {isolated} singularities at $0$ (i.e., locally bounded outside $0$) will be called {\it weights}, and their collection is denoted by $\Wx$. The Monge--Amp\`ere operator $(dd^c)^n$ is well defined on such functions; the {\it residual mass} of $\vph\in\Wx$ is
$(dd^c\vph)^n(\{0\})$.
Note that a solution to the Monge--Amp\`ere equation $(dd^cu)^n=\tau\delta_0$ need not have isolated singularity at $0$, see Example~\ref{ex:nonisolated}. Nevertheless, in this paper we restrict ourselves to the isolated singularities as the simplest and yet interesting case.

A weight $\vph\in \Wx$ is  called {\it maximal} if it is a maximal psh function on a punctured
neighborhood of the origin (satisfies $(dd^c\vph)^n=0$ outside $0$). A basic example is $\vph=\log|F|$ for an equidimensional holomorphic mapping $F$ with isolated zero at $0$, and in this case the residual mass of $\log|F|$ equals the multiplicity of $F$. The collection of all maximal weights will be denoted by $\MW$, and ${\rm cl}(\MW)$ is the collection of all {\it maximal singularities}.

\medskip

By Demailly's Comparison Theorem \cite{D1}, the inequality $\vph\le\psi$ for two weights $\vph,\psi\in\Wx$ implies $(dd^c\vph)^n(\{0\})\ge(dd^c\psi)^n(\{0\}$. A partial converse to this statement is the following Domination Principle.

\begin{lemma}\label{lem:zi} {\rm (\cite[Lemma 6.3]{R7}, \cite[Thm. 3.6]{ACCP})} Let $D$ be a bounded hyperconvex
domain and let $\vph_1, \vph_2\in \PSH(D)\cap
L_{loc}^{\infty}(D\setminus\{0\})$ be two functions such that $(dd^c\vph_i)^n=\tau\delta_0$ and $\vph_i|_{\partial D}=0$.
If $\vph_1\ge \vph_2$ in $D$, then $\vph_1=\vph_2$.
\end{lemma}

The following results on convergence of the residual masses of monotone sequences of maximal weights will be repeatedly used in the paper.

\begin{lemma}\label{lem:conv} Let $D$ be a bounded hyperconvex domain containing $0$, and let functions $\vph_j,\vph\in\PSH^-(D)$, $j=1, 2, \ldots$, be locally bounded and maximal on $D\setminus\{0\}$, and equal to $0$ on $\partial D$.
\begin{enumerate}
\item[(i)] If $\vph_j\ge\vph$ decrease to a function $\psi$, then $\psi=\vph$ iff $(dd^c\vph_j)^n(\{0\})\to(dd^c\vph)^n(\{0\})$;
\item[(ii)] If $\vph_j\le\vph$  increase to a function $\eta$, then $\eta^*=\vph$ iff $(dd^c\vph_j)^n(\{0\})\to(dd^c\vph)^n(\{0\})$;
    \end{enumerate}
here $\eta^*(x)=\limsup_{y\to x}\eta(y)$ is the upper semicontinuous regularization of $\eta$.
\end{lemma}

\begin{proof} {\sl (i) } Since the Monge--Amp\`ere operator is continuous with respect to decreasing sequences, the function $\psi$ is maximal on $D\setminus\{0\}$ and $(dd^c\vph_j)^n(\{0\})\to(dd^c\psi)^n(\{0\})$. Assuming $(dd^c\psi)^n(\{0\})=(dd^c\vph)^n(\{0\})$, we get the two functions, $\psi\ge\vph$, maximal on $D\setminus\{0\}$, equal to $0$ on $\partial D$, and with the same residual Monge--Amp\`ere mass at $0$. By the Domination Principle, $\psi=\vph$.

{\sl (ii) } Similar proof, the relation  $(dd^c\vph_j)^n(\{0\})\to(dd^c\eta^*)^n(\{0\})$ in this case being
due to Bedford--Taylor's result for increasing sequences of bounded psh functions, applied to the functions $\psi_{j}=\max\{\vph_j,-1\}$, since  $(dd^c\vph_j)^n(\{0\})=(dd^c\psi_{j})^n(D)$. (Alternatively, one can refer to the convergence results in the Cegrell class \cite{C2}, \cite{X}.)
\end{proof}

\medskip

Both the Domination Principle and Lemma \ref{lem:conv} can be deduced from the following result.

\begin{lemma}\label{lem:ACCP} {\rm (\rm\cite[Prop. 3.4]{NP})} Let $u,v$ be from the Cegrell class $\mathcal F$ in a bounded hyperconvex domain $D$, and $u\le v$ in $D$. Then
$$\int_D(v-u)^n (dd^cw)^n\le n!\,\int_D w\left[(dd^cu)^n-(dd^cv)^n\right]$$
for any $w\in\PSH(D)$, $0\le w\le 1$.
\end{lemma}


\subsection{Relative types and Lelong numbers}

For any function $u\in \PSH_0$,  its {\it type
relative to a weight} $\vph\in \MW$ \cite{R7} is
\beq\label{eq:reltype} \sigma(u,\vph)=\liminf_{z\to
0}\frac{u(z)}{\vph(z)}=\lim_{r\to-\infty}r^{-1}\Lambda(u,\vph,r),\eeq
where
\beq\label{eq:Lambda} \Lambda(u,\vph,r)=\sup\{u(z):\vph(z)<r\}.
\eeq
The type assumes finite, nonnegative values, and the maximality of $\vph$
implies the bound
\beq\label{eq:rtbound}u\le\sigma(u,\vph)\vph+O(1).\eeq
Evidently, $\sigma(u,\vph)=\sigma(v,\psi)$ for any $v\in {\rm cl}(u)$, $\psi\in {\rm cl}(\vph)$, so the relative type is a function on singularities.
When $\vph(z)=\log|z|$, the type $\sigma(u,\vph)$ is the Lelong number $\nu(u,0)$ of $u$ at $0$.
If $f$ is an analytic function and $F$ is an equidimensional holomorphic mapping, then, as follows from \cite{LT}, $\sigma(\log|f|,\log|F|)$ equals the {\it Samuel asymptotic function} \cite{Sa}
\beq\label{eq:Sam}\bar\nu_\cI(f)=\lim_{k\to\infty}k^{-1}\max\,\{m\in\Z_+:\:f^k\in \cI^m\}\eeq
for the ideal $\cI$ generated by the components of the mapping $F$.

Given an irreducible analytic variety $M\subset X$, the value
\beq\label{eq:gLel} \nu(u,M)=\inf\,\{\nu(u,y):\:y\in M\}\eeq
is the {\it generic Lelong number of $u$ along $M$}. By Siu's theorem,
$\nu(u,x)=\nu(u,M)$ for all $x\in M\setminus M'$, where $M'$ is a countable union of proper analytic subsets of $M$.


\subsection{Green and Green-like functions}\label{ssec:Greenf}
We will use the following extremal function introduced by V.
Zahariuta \cite{Za0}, see also \cite{Za}. For a
bounded hyperconvex domain $D$ containing $0$,  $PSH^-(D)$ will mean the collection of all negative psh functions on $D$.
Given a weight $\vph\in \MW$, denote
$$
G_\vph(z)=G_{\vph,D}(z)=\sup\{u(z):\: u\in PSH^-(D),\ \sigma(u,\vph)\ge 1\}.$$
The function is psh in $D$,
maximal in $D\setminus \{0\}$, $G_{\vph}\in{\rm cl}(\vph)$,
and $G_{\vph}= 0$ on $\partial D$; moreover, it is a
unique function with these properties.
Furthermore, if $\vph$ is continuous near $0$, then $G_{\vph}$ is continuous on $D$ \cite[Prop.~181]{Za0}, see also \cite[Thm.~1.4.6]{Za}; the continuity of $\phi\in\Wx$ is understood here as continuity of $\exp\phi$. We
will refer to this function as the {\it Green} {\it function
with singularity} $\vph$.
If $\vph(z)=\log|z|$, then  $G_{\vph}$ is the standard pluricomplex Green function $G_{0,D}$ of $D$ with pole at $0$.

Several 'maximization' procedures of non-maximal psh functions are described in \cite{R7}.
The notion of Green function $G_\vph$ is extended to arbitrary weights $\vph\in\Wx$ as
\beq\label{eq:cgf} G_\vph(z)=G_{\vph,D}(z) = \limsup_{y\to z}\sup\{v(y):\: v\in\PSH^-(D),\ v\le \vph+O(1)\},\eeq
the {\it complete greenification} of $\vph$.
It is a psh function in $D$, maximal on $D\setminus \{0\}$, equal to zero on $\partial D$; evidently, $\sigma(\vph, G_\vph)\ge 1$.
When $\vph\in\Wx\setminus\MW$, the singularity of $G_\vph$ can differ from that of $\vph$; nevertheless, the relative types and the residual Monge--Amp\`ere mass remain the same: $\sigma(\vph,\psi)=\sigma(G_\vph,\psi)$ for every $\psi\in\MW$ and $ (dd^c\vph)^n(\{0\})=(dd^cG_{\vph})^n(\{0\})$.

Furthermore, given a function $u\in\PSH_0$ and an arbitrary collection $P\subset\MW$,
\beq\label{eq:genind0}
h_{u}^P(z)=h^P_{u,D}(z)=\sup\,\{v(z):\:v\in PSH^-(D),\ \sigma(v,\psi) \ge \sigma(u,\psi)\ \forall\psi\in
P\}
\eeq
is the {\it type-greenification} of $u$ with respect to the collection $P$ \cite{R7}. When $\vph\in\Wx$, the function $h_{\vph}^P$ belongs to $\MW$, equals $0$ on $\partial D$, and satisfies $\sigma(h_{\vph}^P,\psi)= \sigma(\vph,\psi)$ for all $\psi\in P$. Such a function can also be represented as the best psh minorant of the family $\{G_\psi:\: \psi\in P,\ \sigma(\vph,\psi)\ge1\}$. The {\sl raison d'\^{e}tre} for the type-greenification is that it gives the best possible bound on $u\in\PSH_0$, $u\le h_{u}^P+O(1)$, when the only available information on $u$ is the values $\sigma(u,\psi)$ for all $\psi\in P$.

The functions are related by $G_\vph=h_{\vph}^{\MW}\le h_{\vph}^{P_1}\le h_{\vph}^{P_2}$ if $P_2\subset P_1\subset\MW$.


\subsection{Multiplier ideals and Demailly's approximations} \label{ssecDem}
Given $u\in\PSH(X)$, let $\cJ(u)$ denote the {\it multiplier ideal sheaf\:} for $u$, that is, the sheaf of ideals of germs $f\in\cO_X$ such that $|f|\,e^{-u}\in L^2_{loc}$. The sheaf is coherent, and for any pseudoconvex $D\Subset X$, the restriction of $\cJ(u)$ to $D$ is generated as $\cO_D$-module by any basis of the Hilbert space
$\cH(u)=\{f\in\cO(D):\:\ |f|\,e^{-u}\in L^2(D)\}$,
see \cite{Na}. A fundamental property of special importance for us is the Subadditivity Theorem \cite[Thm.~2.6]{DEL}
\beq\label{eq:SAT}\cJ(u+v)\subseteq\cJ(u)\cdot\cJ(v)\quad\forall u,v\in\PSH(X).\eeq
We refer to \cite{D6}, \cite{La} for detailed information on multiplier ideals and their applications to analysis and algebraic geometry.

The notion was used by Demailly for approximation of arbitrary psh functions by ones with analytic singularities.
Demailly's Approximation Theorem \cite{D2} says that, given a psh function $u$  on a bounded pseudoconvex domain $D$,  the functions
\beq\label{eq:appr1}\Df_ku=\frac1k\sup\{\log|f|:\:f\in\cO(D),\ \int_D |f|^2e^{-2ku}\,dV<1\}=\frac1{2k}\log\sum_{i=1}^\infty|f_{k,i}|^2,\eeq
where $\{f_{k,i}\}_i$ is an orthonormal basis for $\cH(ku)$, have the bounds
\beq\label{eq:Dap}
u(z)-\frac{C}{k}\le \Df_ku(z)\le \sup_{|\zeta-z|<r}u(\zeta)+\frac1k\log\frac{C}{r^n},\quad z\in D,
\eeq
with a constant $C>0$  independent of $u$ and $k$, provided $\{|\zeta-z|<r\}\subset D$.

By the strong Noetherian property of the coherent sheaf $\cJ(u)$, for any $D'\Subset D$ there exist finitely many functions $f_{k,i}$, $1\le i\le i_0=i_0(ku)$, such that
\beq\label{eq:appr3}\Df_ku(z)= \frac1{2k}\log\sum_{i=1}^{i_0}|f_{k,i}(z)|^2+O(1),\quad z\in D'.\eeq
As a consequence of (\ref{eq:Dap}), $\Df_ku$ converge to $u$ pointwise and in $L_{loc}^1$, and their Lelong numbers at any $x\in D$ tend to the Lelong numbers of $u$.


\section{Asymptotically analytic singularities}

\begin{definition}
A function $u\in \PSH_0$ has {\it  analytic singularity} if
$c\log|F|\in {\rm cl}(u)$ for some $c>0$ and a
holomorphic mapping $F$ from a neighborhood of $0$ to some $\C^N$. The collection of all {\it weights with analytic singularities} at $0$ (and thus $N\ge n$) is denoted by $\AW$.
\end{definition}

\begin{example} By (\ref{eq:appr3}), Demailly's approximants ${\Df_k u}$ have analytic singularities.
\end{example}

\begin{proposition}\label{prop:anmax}
Any isolated analytic singularity is maximal.
\end{proposition}

\begin{proof} Let $\vph=c\log|F|$. As is known, one can always find $n$
functions $\xi_1,\ldots, \xi_n$ (generic linear combinations of the components $F_j$ of
$F$) such that $\log|F|=\log|\xi|+O(1)$; the functions generate a so-called minimal reduction of the ideal $\cI(F_1,\ldots,F_N)$ \cite{NR}. By King's formula, $(dd^c\log|\xi|)^n=0$ outside the zero set of $F$. Therefore, $\vph\sim c\log|\xi|\in\MW$.
\end{proof}

\medskip

Nevertheless, when operating with a sequence of analytic weights, one should be careful in order to get the limit weight to be equivalent to a maximal one. This issue will be essential, for example, in treating inf/sup-analytic singularities in Sections~4 and 5.

\medskip
We are going to deal with singularities that are 'close' to analytic ones. The first class of such weights is defined as follows.

\begin{definition}\label{def:aas}A function  $\psi\in \PSH_0$ has {\it asymptotically  analytic singularity} if for every $\epsilon>0$ there exists a function $\psi_\epsilon$ with analytic singularity at $0$ such that
\beq\label{eq:ass}(1+\epsilon)\psi_\epsilon +O(1)\le  \psi\le (1-\epsilon)\psi_\epsilon +O(1).\eeq
The collection of all {\it asymptotically  analytic weights} will be denoted by $\AAW$.
\end{definition}

\medskip

Unlike analytic singularities, asymptotically analytic ones need not be maximal.  Take, for example,  $\psi(z)=\log|z|-|\log|z||^{1/2}$, then $(1+\epsilon)\log|z|-C_\epsilon\le\psi(z)\le\log|z|$, while its Green function $G_\psi$ in the unit ball is $\log|z|$; note however that $ \psi(z)/G_{\psi}(z)\to 1$ as $z\to 0$. Actually, this holds true for any asymptotically analytic weight:

\begin{proposition}\label{prop:nonmax} Let $G_{\psi}$ be the Green function (\ref{eq:cgf}) of a singularity $\psi\in\AAW$ for a bounded hyperconvex neighborhood $D$ of the origin. Then there exists the limit
$$\lim_{z\to 0}\frac{\psi(z)}{G_{\psi}(z)}=1.$$
As a consequence, for any $u\in\PSH_0$ one can set
$\sigma(u,\psi)=\sigma(u,G_{\psi})$.
\end{proposition}

\begin{proof} Given $\epsilon\in(0,1)$, choose a weight $\psi_\epsilon\in\AW$ and a constant $C_\epsilon>0$ such that
$ (1+\epsilon)\psi_\epsilon -C_\epsilon\le  \psi\le (1-\epsilon)\psi_\epsilon + C_\epsilon$.
Then
$ (1+\epsilon)G_{\psi_\epsilon}\le  G_{\psi}\le (1-\epsilon)G_{\psi_\epsilon}$,
and thus
$$\frac{1-3\epsilon}{1+\epsilon}\le \frac{\psi(z)}{G_{\psi}}\le \frac{1+4\epsilon}{1-\epsilon}$$
for all $z$ sufficiently close to $0$ because $\psi_\epsilon=G_{\psi_\epsilon}+O(1)$ by Proposition~\ref{prop:anmax}.
\end{proof}

\begin{example}\label{ex:tame}
According to \cite{BFaJ}, a continuous weight $\vph\in \Wx$ is called {\it
tame} if there exists a constant $C>0$ such that for every $t>C$ and
every analytic germ $f$ from the multiplier ideal ${\cal J}(t\vph)$
of $t\vph$ at $0$, one has $\log|f|\le (t-C)\vph+O(1)$.
Demailly's approximants $\Df_k\vph$ (\ref{eq:appr1}) of a tame weight $\vph$ satisfy, by \cite[Lemma 5.9]{BFaJ},  \beq\label{eq:Dapt}  \vph+O(1)\le\Df_k\vph\le (1-C_{\vph}/k)\vph+O(1)\eeq
near $0$; moreover, conditions (\ref{eq:Dapt}) characterize tame weights. Therefore, all tame weights are asymptotically  analytic.
\end{example}

\begin{example}
In particular, any exponentially H\"older continuous weight $\vph$, that is, satisfying for some $\beta>0$ the relation
\beq\label{eq:holder} e^{\vph(y)}-e^{\vph(z)}\le |y-z|^\beta\quad{\rm near\ }x,
\eeq
is tame \cite[Lemma 5.10]{BFaJ}. In particular, all weights with analytic singularities are tame.
\end{example}

\begin{example}\label{ex:mc}
Any {\it multicircular} weight $\vph\in\MWO$ (depending on $|z_1|,\ldots,|z_n|$ only) has asymptotically analytic singularity, which can be shown as follows. Since such a function $\vph$ is equivalent to its {\it indicator at $0$}, that is, the Green function for $\vph$ in the unit polydisk $\D^n$ \cite{LeR}, we can assume $\vph=G_\vph$. This implies $\vph(A_cz)=c\,\vph(z)$ for every $c>0$, where $$A_c(z)=(|z_1|^c,\ldots,|z_n|^c),\quad z\in\D^n.$$ Then the function $\tilde\vph(t)=\vph(e^{t_1},\ldots,e^{t_n})$ is convex and positive homogeneous in $\Rnm$, equal to zero on $\partial\Rnm$, and thus is the restriction to $\Rnm$ of the support function of a convex set $\Gamma\subset\Rnp$:
$$\tilde\vph(t)=\sup\{\langle a,t\rangle: a\in\Gamma\}, \quad t\in\Rnm.$$ Therefore, for any $\epsilon>0$ there exist positive integers $m$ and $N$ and monomials $z^{k(j)}$, $k(j)\in\Z_+^n$, $1\le j\le m$, such that
$|\vph(z)-\vph_\epsilon(z)|<\epsilon/2$ for all $z$ with $-1\le\vph(z)\le 0$, where
$$\vph_\epsilon(z)=N^{-1}\max_j\log|z^{k(j)}|.$$
Take any $w\in\D^n$ with $\vph(w)=t<-1$ and let $z=A_cw$ with $c=1/|2t|$. Then $\vph(z)=-1/2$, $A_1w=A_{|2t|}z$, and
$$|\vph(w)-\vph_\epsilon(w)|=|\vph(A_1w)-\vph_\epsilon(A_1w)|=|\vph(A_{|2t|}z)-\psi_\epsilon(A_{|2t|}z)|<|t|\epsilon.
$$
Since $\vph(w)=t$, this implies
$ (1+\epsilon)\vph(w)\le\vph_\epsilon(w)<(1-\epsilon)\vph(w)$
for all $w$ with $\vph(w)<-1$, so $\vph$ is asymptotically analytic. (Actually, it can be shown even to be tame.)
\end{example}

In Example \ref{ex:mc}, the condition of {\sl isolated} singularity is quite important, since otherwise a maximal multicircular function need not be asymptotically analytic, even if its Monge--Amp\`ere mass is supported by the origin.

\begin{example}\label{ex:nonisolated} Here we construct a multicircular function from the Cegrell class (and so, inside the definition domain of the Monge--Amp\`ere operator), solving the equation $(dd^c\Psi)^n=2\delta_0$, and whose singularity is not asymptotically analytic.

For any $N\ge 2$, the function
$$
f_N(r)=\left\{\begin{array}{ll}
(r+1)^{-2}, &
\mbox{$0<r\le N-1$}\\
-2N^{-3}(r+1-\frac32N), &
\mbox{$N-1<r\le\frac32N-1$}\\
0, &
\mbox{$r>\frac32N-1$}\end{array}\right.
$$
is convex on $(0,\infty)$ (the nonzero linear piece is just a segment of the tangent to the curve $y=(r+1)^{-2}$ at $r=N-1$).  When $N\to\infty$, the sequence $f_N$ increases to $f(r)=(r+1)^{-2}$.
Then the support functions $\psi_N(t)=\sup\{\langle a,t\rangle: a\in\Gamma_N\}$ to the convex sets
$\Gamma_N=\{a\in\Rnpt:\: a_2>f_N(a_1)\}$ decrease to the support function $\psi(t)$ for the set $\Gamma=\{a\in\Rnpt:\: a_2>f(a_1)\}$, $t\in{\Bbb R}_-^2$. The corresponding indicators $\Psi_N(z)=\psi_N(\log|z_1|,\log|z_2|)\in\MWO$ decrease then to the indicator $\Psi(z)=\psi(\log|z_1|,\log|z_2|)$. Since $\Psi_N=0$ on $\partial\D^2$ and the total Monge--Amp\`ere mass $(dd^c\Psi_N)^2(\D^2)$ equals twice the volume of the set $\Rnpt\setminus\Gamma_N$ (see, e.g., \cite{R3}) and thus is dominated by $2\,{\rm Vol}(\Rnpt\setminus\Gamma)=2$, the function $\Psi$ belongs to the Cegrell class $\F$, see \cite{C2}. By the monotone convergence theorem for $\F$, $(dd^c\Psi)^2=2\delta_0$, so $\Psi$ is maximal outside the origin.

For any $z_1\neq 0$ and $N\ge 2$, $$\Psi(z_1,0)\le \Psi_N(z_1,0)\le\Psi_N(z_1,z_1^N)=\psi_N(\log|z_1|,N\log|z_1|)\le N^{1/3}\log|z_1|.$$ Therefore, $\Psi(z_1,0)=-\infty$. At the same time, its Lelong numbers outside the origin equal zero (because $(dd^c\Psi_N)^2=0$ there). That is why any analytic weight dominating $\Psi$ must be finite outside the origin, so $\Psi$ is not asymptotically analytic.

\end{example}

{\it Remark.} We have no example of a maximal weight that is not asymptotically analytic.

\section{Green functions of Demailly's approximants}\label{sec:_GZf}

In what follows, $D$ is a bounded hyperconvex domain containing $0$. Given $\phi\in \MW$, let $G_\phi$ denote the Green function of $D$ for the singularity $\phi$. By Proposition~\ref{prop:anmax}, the Green functions $G_{\Df_k\phi}$ of Demailly's approximants $\Df_k\phi$ satisfy $G_{\Df_k\phi}\sim\Df_k\phi$.

\subsection{General case: $\phi\in\MW$}\label{sec:greendem}

By Demailly's Approximation Theorem, the functions $\Df_k\phi$ converge to $\phi$ in $L_{loc}^1$ for any $\phi\in\Wx$, and $G_{\Df_k\phi}\ge G_\phi$. We do not know if this implies convergence of $G_{\Df_k\phi}$ to $G_\phi$, even when $\phi$ is a maximal weight. This is certainly not so if $\phi$ has zero Lelong number (because then $G_{\Df_k\phi}=0$), however existence of such a weight is equivalent to the aforementioned Zero Lelong Number Problem, see \cite[Remark 6.2]{R7}.

\begin{proposition}\label{lem:1} If $\phi\in \MW$, then the sequence of the Green functions $G_{\Df_{m!}\,\phi}$ for the singularities $\Df_{m!}\,\phi$ decreases to the function $\widetilde G_\phi =\inf_k G_{\Df_k\phi}\ge G_\phi$, psh in $D$ and maximal on $D\setminus\{0\}$.
\end{proposition}

\begin{proof} Let $\cJ(k\phi)$ be the multiplier ideals for the functions $k\phi$, $k\in\Z_+$.
By (\ref{eq:SAT}),
\beq\label{eq:cJmk}\cJ((mk)\phi)\subseteq\cJ(k\phi)^m,\quad \forall m,k\in\Z_+.\eeq
Since the multiplier ideals $\cJ(k\phi)$ generate the Demailly approximants $\Df_k\phi$,  we have $\Df_{mk}\phi\le \Df_k\phi+ C(k,m)$ and so,
\beq\label{eq:GDmk} G_{\Df_{mk}\phi}\le G_{\Df_k\phi},\eeq
which implies that the sequence $G_{\Df_{m!}\,\phi}$ is decreasing to some function $\widetilde G_\phi\in\PSH(D)$, maximal in $D\setminus\{0\}$. Moreover, (\ref{eq:GDmk}) yields  $G_{\Df_{m!}\,\phi}\le G_{\Df_k\phi}$ for all $k\le m$, which gives us
 $\widetilde G_\phi=\inf_k G_{\Df_k\phi}=\lim_{m\to\infty}G_{\Df_{m!}\,\phi}\ge G_\phi$ and thus is a maximal weight.
\end{proof}
\medskip

{\it Remark.} Note that (\ref{eq:SAT}) implies $\cJ((k+m)\phi)\subseteq\cJ(k\phi)\cJ(m\phi)$ for all $k,m\in\Z_+$, a stronger relation than (\ref{eq:cJmk}),  however this does not result (at least, immediately) in the relation $G_{\Df_{k+m}\phi}\le G_{\Df_k\phi}+ G_{\Df_m\phi}$, which would give us convergence of $G_{\Df_k\phi}$. Moreover, we do not know if $G_{\Df_k\phi}\to\widetilde G_\phi$ for arbitrary $\phi\in\MW$.

\medskip

The function $\widetilde G_\phi$ has a nice interpretation in terms of greenifications. Namely, let $h_\phi^A$ denote the type-greenification (\ref{eq:genind0}) with respect all analytic singularities:
\beq\label{eq:hA} h_\phi^A=\sup\{v\in\PSH^-(D):\: \sigma(v,\vph)\ge\sigma(\phi,\vph)\quad\forall\vph\in\AW\}.\eeq

\begin{proposition}\label{prop:hA} $\widetilde G_\phi=h_\phi^A$ for any $\phi\in\Wx$. As a consequence, there exists a sequence $\vph_j\in\AW$ such that the Green functions $G_{\vph_j}$ decrease to $h_\phi^A$ as $j\to\infty$.
\end{proposition}

\begin{proof}
As was mentioned in Section~\ref{ssec:Greenf}, the function $h_\phi^A$ is the best psh minorant of the family $\{G_\psi:\:\psi\in\AW,\ \sigma(\phi,\psi)\ge1\}$. Since all the functions $G_{\Df_k\phi}$ belong to that family, $h_\phi^A\le G_{\Df_k\phi}$ and thus $h_\phi^A\le \widetilde G_\phi$. On the other hand, the relation $\sigma(\phi,\psi)\ge1$ means $\phi\le\psi+O(1)$ and thus implies $\Df_k\phi\le\Df_k\psi+O(1)$, so $G_{\Df_k\psi}\ge G_{\Df_k\phi}\ge \widetilde G_\phi$ for all $k$. As $\psi\in\AW$, relation (\ref{eq:Dapt}) implies $G_{\Df_k\psi}\to G_\psi$, which gives $G_\psi\ge \widetilde G_\phi$ for all $\psi\in\AW$ with $\sigma(\phi,\psi)\ge1$, so $h_\phi^A\ge \widetilde G_\phi$.
\end{proof}
\medskip

Proposition~\ref{prop:hA} means that the singularity of $\widetilde G_\phi$ is the best upper bound on the singularity of $\phi$ when the latter can be 'tested' on all analytic weights. In other words, the valuative transforms \cite{BFaJ} of $\widetilde G_\phi$ and $G_\phi$ coincide. This gives us one more motivation on the problem if $\widetilde G_\phi$ equals $G_\phi$.

It turns out that the relation $\widetilde G_\phi =G_\phi$ is completely controlled by the behavior of the Monge--Amp\`ere masses of the functions $\Df_k\phi$.

\begin{proposition}\label{prop:convmasses}  $\widetilde G_\phi =G_\phi$ if and only if  $\,\sup_k (dd^c\Df_k\phi)^n(\{0\})= (dd^c\phi)^n(\{0\})$.
\end{proposition}

\begin{proof} This follows from Proposition~\ref{lem:1}  and Lemma~\ref{lem:conv}(1).
\end{proof}


\subsection{Green functions for asymptotically analytic weights}

For a tame weight $\vph$, inequalities (\ref{eq:Dapt}) imply
\beq\label{eq:Dapt1} G_\vph\le G_{\Df_k\vph}\le (1-{C_\vph}/{k})G_\vph\eeq
and, therefore, uniform convergence of $G_{\Df_k\vph}/G_\vph$ to $1$. We will prove that such a convergence holds true for any asymptotically analytic weight as well.

\begin{theorem}\label{theo:AAMW} Let $\psi\in\MW$. Then $\psi\in\AAW$ if and only if
\beq\label{eq:unif}\frac{G_{\Df_k\psi}}{G_\psi}\to 1 \ {\rm uniformly\ on\ }D\setminus\{0\}.\eeq
 For such a weight $\psi$, we get thus $(dd^c{\Df_k\psi})^n(\{0\})\to(dd^c\psi)^n(\{0\})$ and $\widetilde G_\psi=G_\psi$.
\end{theorem}

{\it Remark.} This gives the equivalence of the conditions (i) and (ii) of Theorem A.

\begin{proof} Relation (\ref{eq:unif}) yields (\ref{eq:ass}) with $\psi_\epsilon= \Df_k\psi$ and $k\ge k(\epsilon)$, so let us prove the reverse implication.
For $\epsilon>0$, let $\psi_\epsilon\in \AW$ be a weight satisfying (\ref{eq:ass}). Then
$$(1+\epsilon)G_{\psi_{\epsilon}}\le G_\psi\le  G_{\Df_k\psi}\le G_{\Df_k(1-\epsilon)\psi_\epsilon}$$ for all $k$.
Therefore,
 \beq\label{eq:111}\frac{G_{\Df_k(1-\epsilon)\psi_\epsilon}(z)}{(1+\epsilon)G_{\psi_\epsilon}(z)}\le \frac{G_{\Df_k\psi}(z)}{G_\psi(z)}\le 1,\quad z\in D\setminus\{0\}.\eeq
Since $(1-\epsilon)\psi_\epsilon$ has analytic singularity, (\ref{eq:Dapt1}) implies
$$  G_{\Df_k(1-\epsilon)\psi_\epsilon}\le (1-C_{\epsilon}/k)(1-\epsilon)G_{\psi_\epsilon}.
$$
Therefore, (\ref{eq:111}) for all $k\ge k(\epsilon)$ gives us
$$1-2\epsilon\le \frac{G_{\Df_k\psi}(z)}{G_\psi(z)}\le 1,\quad z\in D\setminus\{0\}
$$
(and thus, the convergence of the residual masses).
\end{proof}

\begin{corollary}\label{cor:cont} The Green function $G_\psi$ of any bounded hyperconvex domain $D$ with singularity $\psi\in\AAW$ is continuous on $D$.
\end{corollary}


\subsection{Green functions for inf-analytic weights}

As is shown in Theorem \ref{theo:AAMW}, asymptotical analyticity is rather a strong property with regard to behavior of the singularities of the Demailly approximants. In order to get the largest class of weights $\phi$ such that $\widetilde G_\phi =G_\phi$, we introduce the following notion.

\begin{definition}\label{def:ia} A weight $\phi\in\MW$ has {\it inf-analytic singularity} if there exists a sequence of weights $\phi_j\in\AW\cap\MW$ decreasing to $\phi$ on a neighborhood of $0$. The class of all such weights will be denoted by $\UAW$.
\end{definition}

Note that the maximality of the weights $\phi_j$ is an essential requirement because any psh function can be realized as the limit of a decreasing sequence of functions with just analytic singularities \cite[Cor. 4.4]{Bl}.

\begin{theorem}\label{theo:uaw} Let $\phi\in\MW$, then the following conditions are equivalent:
\begin{enumerate}\item [(i)] $\phi\in\UAW$;
 \item[(ii)] $G_\phi =\inf_k G_{\Df_k\phi}$ for any bounded hyperconvex neighborhood $D$ of $0$;
 \item[(iii)] $(dd^c\phi)^n(\{0\})=\inf_k(dd^c\Df_k\phi)^n(\{0\})$;
 \item[(iv)] there exist functions $\psi_j\in\AW$, $\psi_j\ge \phi$, such that $(dd^c\psi_j)^n(\{0\})\to(dd^c\phi)^n(\{0\})$.
\end{enumerate}
\end{theorem}

{\it Remark.} This establishes the equivalence of the first three conditions in Theorem B.

\begin{proof}
Let $\phi\in\UAW$ and let $\phi_j$ be the corresponding maximal analytic weights. By the monotone convergence theorem, $(dd^c\phi_j)^n\to (dd^c\phi)^n$, and the maximality of $\phi_j$ implies the convergence of the residual masses. Since the sequence $G_{\phi_j}$ decreases to a function $u\ge G_\phi$, Lemma~\ref{lem:conv}(1) gives us then the equality $u=G_\phi$. Note now that for any $j$ we have $\inf_k G_{\Df_k\phi}\le \inf_k G_{\Df_k{\phi_j}}=G_{\phi_j}$, which shows { (i) }$\Rightarrow${ (ii)}.

To prove the implication {(ii)\ }$\Rightarrow${ (i)}, we chose $D=\{\phi<c\}$ for some $c$ sufficiently close to $-\infty$, so $\phi=G_\phi +c$ on $D$, and apply Proposition~\ref{lem:1}.

The relation {(ii) $\Leftrightarrow$ (iii)} follows from Proposition~\ref{prop:convmasses}.

Condition {(i)} obviously implies {(iv)} with $\psi_j=\phi_j$.

Finally, assume {(iv)}.
Then for any $m,j\in\Z_+$, we have
$ G_{\Df_{m!}\,\phi}\le G_{\Df_{m!}\,\psi_j}$, so
 $$ (dd^c\Df_{m!}\,\phi)^n(\{0\})\ge (dd^c\Df_{m!}\,\psi_j)^n(\{0\}).$$
Given $\epsilon>0$, choose $j$ such that $(dd^c\psi_j)^n(\{0\})>(dd^c\phi)^n(\{0\})-\epsilon/2$. Since $\psi_j\in\AW$, we can find then, by Theorem~\ref{theo:AAMW}, a number $m_0$ such that
$$(dd^c{\Df_{m!}\,\psi_j})^n(\{0\})\ge (dd^c\psi_j)^n(\{0\})-\epsilon/2$$ for all $m\ge m_0$. Therefore, $$(dd^c{\Df_{m!}\,\phi})^n(\{0\})\ge (dd^c\phi)^n(\{0\})-\epsilon, \quad m\ge m_0,$$ and Propositions~\ref{lem:1} and \ref{prop:convmasses} prove {(ii)}.
\end{proof}

\medskip

\begin{corollary} $\AAW\cap\MW\subset\UAW$. \end{corollary}

\begin{proof} This follows from Theorems \ref{theo:AAMW} and \ref{theo:uaw}. \end{proof}

\begin{corollary} If $\phi\in\UAW$, then $(dd^c{\Df_{k}\,\phi})^n(\{0\})\to (dd^c\phi)^n(\{0\})$ and $G_{\Df_k\phi}\to G_\phi$ in $L^n(D)$ as $k\to\infty$.
\end{corollary}

{\it Remark.} This is the last assertion of Theorem B.

\begin{proof} Let $\phi_j\ge\phi$ be maximal analytic weights from the definition of inf-analyticity.
We have then, for any $j$,
 $$(dd^c\phi_j)^n(\{0\})=\lim_{k\to\infty}(dd^c{\Df_{k}\,\phi_j})^n(\{0\})\le\liminf_{k\to\infty}(dd^c{\Df_{k}\,\phi})^n(\{0\}),$$ so the relations $(dd^c\phi_j)^n(\{0\})\to (dd^c\phi)^n(\{0\})$ and $(dd^c{\Df_{k}\,\phi})^n(\{0\})\le(dd^c\phi)^n(\{0\})$ prove the first convergence.

Since $G_{\Df_k\phi}\ge G_\phi$ and these functions are maximal on $D\setminus\{0\}$ and equal to zero on $\partial D$, the second statement follows from Lemma~\ref{lem:ACCP}.
\end{proof}

\medskip
In view of Proposition~\ref{prop:hA}, we have

\begin{corollary}\label{cor:hA1} A weight $\phi\in\MW$ belongs to $\UAW$ if and only if $G_\phi$ coincides with the function $h_\phi^A$ defined by (\ref{eq:hA}).
\end{corollary}

{\it Remark.} In terms of formal psh functions on the space $\mathcal V$ of valuations on the ring $\cO_0$ introduced in \cite{BFaJ}, the corollary can be stated as follows: a maximal  weight $\phi$ is inf-analytic if and only if $G_\phi=\sup\,\{v\in\PSH^-(D):\: \hat v\le\hat\phi\}$, where $\hat v$, $\hat\psi$ are the valuative transforms of $v$ and $\psi$.

Furthermore, a weight $\phi\in\UAW$ is asymptotically analytic if and only if $\hat\phi$ is continuous on $\mathcal V$, which follows from Dini's lemma applied to the valuative transforms $\hat\phi_j$ of the analytic weights $\phi_j$ from Definition~\ref{def:ia}.


\section{Asymptotic multiplier ideals}\label{sec:multideals}

Another properties of analytically approximable weights come from consideration of asymptotic multiplier ideals, the notion introduced in \cite{ELS1}. Recall that a family $\frad$ of ideals ${\mathfrak a}_k\subset\cO_0$ is called {\it graded} if $\fra_m\cdot\fra_k\subseteq\fra_{m+k}$ for all positive integers $m$ and $k$; we assume $\fra_k\neq\{0\}$ for $k>0$. If all $\fra_k$ are primary (i.e, $V(\fra_k)=\{0\}$), then there exists the limit
 \beq\label{eq:lima}e(\frad)=\lim_{k\to\infty}k^{-n}e(\fra_k),\eeq
 called the {\it multiplicity} of $\frad$ \cite[Cor.~1.5]{Mustata}; here $e(\fra_k)$ is the Samuel multiplicity of $\fra_k$.

 The multiplier ideal of an ideal $\fra$ is understood here as the multiplier ideal $\cJ(\log|F|)$, where $F$ is a mapping whose components generate the ideal $\fra$, see Section~\ref{ssecDem}.
 As follows from the Subadditivity Theorem, the family of multiplier ideals $\{\cJ(\fra_p^{k/p})\}_{p\in{\Bbb N}}$ has a unique maximal element, denoted here by $\frj_k$, and it coincides with $\cJ(\fra_p^{k/p})$ for all sufficiently great values of $p$, see \cite{ELS1}. The ideals $\frj_\bullet=\{\frj_k\}$ are called {\it asymptotic multiplier ideals} of $\fra_\bullet$. One has always the inclusions
\beq\label{eq:asmult}\frj_{km}\subseteq \frj_k^m\eeq
and \beq\label{eq:asmult1}\fra_k\subseteq \frj_k,\eeq see \cite[Prop. 1.7]{ELS1}; furthermore, as shown in \cite{Mustata}, there exists the limit
\beq\label{eq:limj} e(\frj_\bullet)=\lim_{k\to\infty} k^{-n}e(\frj_k).\eeq
In some cases, $\frj_\bullet$ are not much bigger than $\fra_\bullet$, in the sense that $e(\frad)=e(\frj_\bullet)$: for instance, when $\fra_k$ are defined by Abhyankar valuations \cite{ELS2} or when they are monomial \cite{Mustata}.

Let us apply the machinery of Green functions to such families of ideals.
Given a bounded hyperconvex neighborhood $D$ of $0$, let $G_{\fra_k}$ denote the Green function of $D$ with singularity along $\fra_k$ \cite{RSig2} (i.e., with the singularity $\vph_k=\log|F_k|$, where $F_k$ are holomorphic mappings whose components generate $\fra_k$), and let $h_k=k^{-1}G_{\fra_k}$. Similarly, we denote $H_k=k^{-1}G_{\frj_k}$; as follows from (\ref{eq:asmult1}),
\beq\label{eq:dom}H_k\ge h_k.\eeq

Since the family $\fra_\bullet$ is graded, we have $\fra_k^m\subseteq\fra_{km}$ and thus  $h_{k}\le h_{km}$, hence we can argue as in the proof of Proposition~\ref{lem:1}. In doing so, we derive that the sequence $h_{m!}$ is increasing to the function $h_{\fra_\bullet}:=\sup_k h_k$ as $m\to\infty$. Denote its upper semicontinuous regularization $(h_{\fra_\bullet})^*$ by $G_{\fra_\bullet}$.

In the same manner, relation (\ref{eq:asmult}) implies $H_{k}\ge H_{km}$, so the sequence $H_{m!}$ is decreasing to the function $G_{\frj_\bullet}:=\inf_k H_k\in\MW$ as $m\to\infty$. By (\ref{eq:dom}),
\beq\label{eq:dom1}G_{\fra_1}\le G_{\fra_\bullet}\le G_{\frj_\bullet}\le G_{\frj_1};\eeq
in particular, $G_{\fra_\bullet},G_{\frj_\bullet}\in\MW$.

\begin{proposition}\label{prop:hkHk}  Let $h_k$, $H_k$, $G_{\fra_\bullet}$, and $G_{\frj_\bullet}$ be defined as above for a graded family $\fra_\bullet$ of primary ideals in $\cO_0$. Then
\begin{enumerate}
\item[(i)] $h_k\to G_{\fra_\bullet}$ and $H_k\to G_{\frj_\bullet}$ in $L^n(D)$ as $k\to\infty$;
\item[(ii)] $G_{\fra_\bullet}=G_{\frj_\bullet}$ if and only if $e({\fra_\bullet})=e({\frj_\bullet})$.
\end{enumerate}
\end{proposition}

\begin{proof}
 Note first that, since the Samuel multiplicity $e(\cI)$ of a primary ideal $\cI$ generated by $f_1,\ldots,f_m$ equals the residual Monge--Amp\`ere masse of the function $\log|f|$  \cite[Lemma 2.1]{D8}, we have
$e(\fra_k)=k^n(dd^c h_k)^n(\{0\})$ and $e(\frj_k)=k^n(dd^c H_k)^n(\{0\})$.

Since $h_{m!}$ increase a.e. to $h_{\fra_\bullet}$ and $H_{m!}$ decrease to $G_{\frj_\bullet}$ as $m\to\infty$, we have
$$(dd^c {h_{m!}})^n(\{0\})\to (dd^c {G_{\fra_\bullet}})^n(\{0\}),\quad (dd^c {H_{m!}})^n(\{0\})\to (dd^c {G_{\frj_\bullet}})^n(\{0\}),$$ so (\ref{eq:lima}) and (\ref{eq:limj}) give us
$$(dd^c h_k)^n(\{0\})\to (dd^c {G_{\fra_\bullet}})^n(\{0\}),\quad (dd^c H_k)^n(\{0\})\to (dd^c {G_{\frj_\bullet}})^n(\{0\});$$
in other words,
\beq\label{eq:multas} (dd^c {G_{\fra_\bullet}})^n(\{0\})=e(\fra_\bullet),\quad
(dd^c {G_{\frj_\bullet}})^n(\{0\})=e(\frj_\bullet).
\eeq
In addition, $h_k\le G_{\fra_\bullet}$ and $H_k\ge G_{\frj_\bullet}$ and all these functions are maximal on $D\setminus\{0\}$, equal to zero on $\partial D$. Therefore, {(i)} follows, as in the proof of Theorem~\ref{theo:uaw}, from Lemma~\ref{lem:ACCP}.

Statement {(ii)} follows now from (\ref{eq:dom1}) and (\ref{eq:multas}) by Lemma~\ref{lem:zi}.
\end{proof}

\medskip

{\it Remark.} A more direct way of showing $h_k\to G_{\fra_\bullet}$ is by observing that the relation $\fra_m\cdot\fra_k\subseteq\fra_{m+k}$ implies $mh_m+kh_k\le (m+k)h_{m+k}$, so that the sequence $h_k$ increases. We did not use that argument for the sake of uniform treating for both $h_k$ and $H_k$.

\medskip

We are going now to find conditions providing the equality $G_{\fra_\bullet}=G_{\frj_\bullet}$.
Since $\frj_k=\cJ(\fra_p^{k/p})=\cJ(kh_p)$ for $p>p(k)$, we have $H_k=\Df_kh_p+O(1)$ for $p>p(k)$. Therefore,
$H_k\le\Df_kG_{\fra_\bullet}+O(1)$.
In view of Theorem~\ref{theo:uaw}, this proves

\begin{proposition}\label{prop:hk} $G_{\fra_\bullet}\le H_k\le G_{\Df_kG_{\fra_\bullet}}$. Therefore,
 $G_{\fra_\bullet}=G_{\frj_\bullet}$ if $G_{\fra_\bullet}\in\UAW$.
\end{proposition}

{\it Remark.} If the ideals $\fra_k$ are monomial, then the limit function $G_{\fra_\bullet}$ is multicircular and thus, as shown in Example \ref{ex:mc}, asymptotically analytic. Applying Proposition~\ref{prop:hk}, we recover the aforementioned result $e({\fra_\bullet})=e({\frj_\bullet})$ for monomial sequences  \cite{Mustata}.

\medskip

From now on, we specify
\beq\label{eq:grtypes}\fra_k=\fra_k(\phi)=\{f\in\cO_0:\: \sigma(\log|f|,\phi)\ge k\},\quad \phi\in\MW.\eeq
Since $\sigma(u+v,\phi)\ge\sigma(u,\phi)+\sigma(v,\phi)$, it is a graded family of ideals. To guarantee that $\fra_k$ are primary and different from $\{0\}$, we assume  $\phi$ to have finite {\it {\L}ojasiewicz exponent}:
\beq\label{eq:Loj}\limsup_{z\to 0} \frac{\phi(z)}{\log|z|}<\infty.\eeq
We will see that, for such a choice of $\fra_k$, the relation $G_{\fra_\bullet}=G_{\frj_\bullet}=G_\phi$ is true for asymptotically analytic and even more general singularities $\phi$.

\medskip

Since $\sigma(h_k,\phi)\ge 1$, we have $h_k\le G_\phi$ and thus
\beq\label{eq:hfgf} G_{\fra_\bullet}\le G_\phi.\eeq

\begin{definition} A weight $\phi\in\MW$ has {\it sup-analytic singularity} if it coincides a.e.\! on a neighborhood of $0$ with the limit of an increasing sequence of maximal analytic weights. The collection of sup-analytic weights is denoted by $\OAW$.
\end{definition}

\begin{proposition} A maximal weight $\phi$ belongs to $\OAW$ if and only if $\phi\ge\psi_j$ for some weights $\psi_j\in\AW$ with $(dd^c\psi_j)^n(\{0\})\to(dd^c\phi)^n(\{0\})$ as $j\to\infty$.
\end{proposition}

\begin{proof} Let $\psi_j$ satisfy the conditions of the proposition, and let $\phi_j$ be the Green functions for the singularities $\max_{k\le j}\,\psi_j\in\AW$ on the domain $\{\phi<c\}$ for some $c\in\R$. They increase a.e. to a function $u\le \phi-c$. Since $(dd^c\phi)^n(\{0\})\le(dd^c\phi_j)^n(\{0\})\le (dd^c\psi_j)^n(\{0\})$, Lemma~\ref{lem:conv}(ii) implies the convergence of the functions $\phi_j+c$ to $\phi$.

The reverse implication is trivial.
\end{proof}

\medskip

Note that any maximal asymptotically analytic weight is sup-analytic. Note also that any weight $\phi\in\OAW$ has finite {\L}ojasiewicz exponent (\ref{eq:Loj}).

\begin{proposition}\label{prop:OAW} $G_{\fra_\bullet}=G_\phi$ if and only if $\phi\in\OAW$.
\end{proposition}

{\it Remark.} This is statement (i) of Theorem C.

\begin{proof}
Let $\phi\in\OAW$.
Given a sequence $\phi_i\in\MW\cap\AW$ increasing a.e.\! to  $\phi$, the weights $\psi_j=\max_{i\le j}\phi_i$ have analytic singularities and satisfy $\phi\ge\psi_j+O(1)$ and $$(dd^c\phi)^n(\{0\})\le(dd^c\psi_j)^n(\{0\})\le(dd^c\phi_j)^n(\{0\}),$$ so $(dd^c\psi_j)^n(\{0\})\to(dd^c\phi)^n(\{0\})$. In addition, the corresponding Green functions $G_{\psi_j}$ increase to some function $\widetilde \psi$. By Lemma~\ref{lem:conv}(ii), $\widetilde \psi^*=G_\phi$.

Denote $\fra_k^{(j)}=\{f\in\cO_x:\: \sigma(\log|f|,\psi_j)\ge k\}$, then $\fra_k^{(j)}\subseteq\fra_k$ and hence $h_k^{(j)}\le h_k$ and $G_{\fra_\bullet}^{(j)}\le G_{\fra_\bullet}$.  Since $\phi_j$ have analytic singularity, $G_{\fra_\bullet}^{(j)}=G_{\phi_j}$ and so, $G_{\fra_\bullet}\ge G_{\phi_j}$ for all $j$ and thus $G_{\fra_\bullet}\ge G_{\phi}$. In view of (\ref{eq:hfgf}), this proves $G_\phi= G_{\fra_\bullet}$.

The other direction is evident by the construction of $G_{\fra_\bullet}$.
\end{proof}

\begin{theorem}\label{theo:hfhf} $G_{\fra_\bullet}=G_{\frj_\bullet}=G_\phi$ if and only if $\phi\in\OAW\cap\UAW$.
\end{theorem}

{\it Remark.} This is statement (ii) of Theorem C.

\begin{proof} If $\phi\in\OAW\cap\UAW$, then, by Proposition~\ref{prop:OAW}, $G_\phi=G_{\fra_\bullet}$. Now Proposition~\ref{prop:hk} becomes applicable and gives $G_{\fra_\bullet}=G_{\frj_\bullet}$.
The converse implication is obvious because, by construction, $G_{\fra_\bullet}\in\OAW$ and $G_{\frj_\bullet}\in\UAW$.
\end{proof}

\medskip
{\it Remarks.} 1. Since every Abhyankar valuation is generated by a special tame weight \cite[Thm.~5.13]{BFaJ}, Theorem~\ref{theo:asmult} extends the result $e(\frad)=e(\frj_\bullet)$  \cite{ELS2} from Abhyankar valuations to a much larger class of weights (note however that, for Abhyankar valuations, a stronger result is proved there).

\medskip
2. By using Theorem~\ref{theo:hfhf}, it is easy to see that $G_{\frj_\bullet}\le \widetilde G_\phi=\inf_k G_{\Df_k\phi}$ for any weight $\phi\in\MW$. We do not know if the sole condition $\phi\in\UAW$ implies $G_{\frj_\bullet}=G_\phi$.

\medskip
Since $\AAW\cap\MW\subset\UAW\cap\OAW$, the 'if' statement of Theorem~\ref{theo:hfhf} holds true for all asymptotically analytic singularities. It is not surprising that for such weights one can claim even more.

We start with a characterization of $\AAW$ in terms of the functions $h_k=k^{-1}G_{\fra_k}$.

\begin{proposition}\label{prop:Gk} Let $\phi\in\MW$, then $\phi\in\AAW$ if and only if
${G_\phi}/{h_k}\to 1$, uniformly on $D\setminus\{0\}$, as $k\to\infty$.
\end{proposition}

\begin{proof}
If $\phi$ has asymptotically analytic singularity, then for any $\epsilon\in(0,1)$ one can find a holomorphic mapping $F_\epsilon$ and a positive integer $k$ such that
\beq\label{eq:Gk1}\frac1k\log|F_\epsilon|\le\phi+O(1)\le \frac{1-\epsilon}k\log|F_\epsilon|,\eeq
so all the components of $F_\epsilon$ belong to $\fra_k$ and thus $h_k\ge (1-\epsilon)^{-1} G_\phi$, which gives the convergence. Conversely, the inequality $ G_\phi/h_k>(1-\epsilon)$ on $D\setminus\{0\}$ implies (\ref{eq:Gk1}).
\end{proof}

\begin{theorem}\label{theo:asmult} If a graded family $\fra_\bullet=\{\fra_k\}$ is defined by (\ref{eq:grtypes}) for $\phi\in\AAW$ and $\frj_k$ are the corresponding asymptotic multiplier ideals, then $G_{\fra_k}/G_{\frj_k}\to 1$ uniformly on $D\setminus\{0\}$ and $G_{\fra_\bullet}=G_{\frj_\bullet}=G_\phi$.
\end{theorem}

{\it Remark.} This is statement (iii) of Theorem C.

\begin{proof} The theorem follows from Propositions~\ref{prop:Gk}, \ref{prop:hk} and Theorem~\ref{theo:AAMW}.
\end{proof}

\medskip
Observe that our approach gives a precise meaning to the fact that the families $\fra_\bullet$ and $\frj_\bullet$ are 'close'. Since all the ideals $\fra_k$ and $\frj_k$ are integrally closed, they are in a one-to-one correspondence with their Green functions $G_{\fra_k}$ and $G_{\frj_k}$ in the sense that, for example, $f\in\fra_k \Leftrightarrow \log|f|\le G_{\fra_k}+O(1)$. These functions have the scaled limits $$G_{\fra_\bullet}=\lim_{k\to\infty} k^{-1} G_{\fra_k},\quad  G_{\frj_\bullet}=\lim_{k\to\infty}k^{-1} G_{\frj_k},$$
so, with an obvious denotation,
$$  \lim_{k\to\infty} k^{-1} \log|\fra_k| =\{u\in\PSH_0:\: \sigma(u,G_{\fra_\bullet})\ge 1\},
$$
$$ \lim_{k\to\infty}k^{-1}\log|\frj_k|=\{u\in\PSH_0:\: \sigma(u,G_{\frj_\bullet})\ge 1\},
$$
and for 'good' weights $\phi$, the limits coincide and equal $\{u\in\PSH_0:\: \sigma(u,\phi)\ge 1\}$.


\section{Relative types}

Let us first mention some simple properties of types relative to analytically approximable weights.

\begin{proposition}\label{cor:convergeps} Let $\psi\in \AAW$ and let $\psi_\epsilon\in \AW$ be weights from (\ref{eq:ass}), then $\sigma(u,\psi_\epsilon)\to\sigma(u,\psi)$ as $\epsilon\to 0$.
\end{proposition}

\begin{proof} By (\ref{eq:ass}),
$(1-\epsilon)\sigma(u,\psi)\le \sigma(u,\psi_\epsilon)\le (1+\epsilon)\sigma(u,\psi)$.
\end{proof}

\begin{proposition}\label{cor:cont1} If $\phi\in\AAW$, then $\sigma(u,\Df_k\phi)\to\sigma(u,\phi)$ for every $u\in\PSH_0$.
\end{proposition}

\begin{proof} This is a direct consequence of Theorem~\ref{theo:AAMW}.
\end{proof}

\medskip
A corresponding, less obvious property for inf-analytic singularities reads as follows.

\begin{proposition}\label{cor:cont2} If $\phi\in\UAW$, then $\sigma(u,\Df_{m!}\,\phi)\searrow\sigma(u,\phi)$ as $m\to\infty$,  for every function $u\in\PSH_0$.
\end{proposition}

\begin{proof}
We denote $\vph_m=G_{\Df_{m!}\,\phi}$. By Proposition~\ref{prop:hA} and Corollary~\ref{cor:hA1}, the sequence $\vph_m$ decreases to $G_\phi$. Therefore, $\sigma(u,\vph_m)$ is a decreasing sequence for any $u\in\PSH_0$. Let us show that its limit $\sigma(u)$ coincides with $\sigma(u,\phi)$.

The functional $\sigma:\PSH_0\to [0,\infty]$ satisfies $\sigma(\max\{u,v\})=\min\{\sigma(u),\sigma(v)\}$, is positive on $\log|z|$, finite on all $u\not\equiv-\infty$, positive homogeneous, and lower semicontinuous because if $u_k\to u$ in $L_{loc}^1$, then for each $m$,
$$ \limsup_{k\to\infty}\sigma(u_k)\le \limsup_{k\to\infty}\sigma(u_k,\vph_m)\le \sigma(u,\vph_m).$$
Therefore, by \cite[Thm. 4.3]{R7}, $\sigma(u)=\sigma(u,\psi)$ for some weight $\psi\in\MW$. Observe that $\sigma(u,\psi)\ge \sigma(u,\phi)$ for all $u$, so $G_\psi\ge G_\phi$. On the other hand, $\sigma(u,\psi)\le \sigma(u,\vph_m)$ for all $m$ and hence $G_\psi\le G_\phi$, which implies $\sigma(u,\psi)=\sigma(u,\phi)$.
\end{proof}

\medskip
More advanced special properties of types relative to such weights are given below.

\subsection{Types of Demailly's approximants}

As stated in Demailly's Approximation Theorem, the Lelong numbers of $\Df_ku$ converge to that of $u$ for every $u\in\PSH_0$. Here we extend this to types of $\Df_ku$ relative to asymptotically analytic weights.

We first relate the $\vph$-types of the functions $\Df_ku$
and $u$ for exponentially H\"older continuous weights $\vph$.

\begin{lemma} \label{prop:appr} Let $\vph\in \MW$ satisfy (\ref{eq:holder}) and let $u\in \PSH_0$. Then the types $\sigma(\Df_ku,\vph)$ of its Demailly approximants $\Df_ku$ satisfy the relations
\beq\label{eq:d2}
\sigma(u,\vph)-\frac{n}{k\beta}\le\sigma(\Df_ku,\vph)\le\sigma(u,\vph).  \eeq
\end{lemma}

\begin{proof} The first inequality in (\ref{eq:Dap}) implies the
relation $\sigma(\Df_ku,\vph)\le\sigma(u,\vph)$.

 If $\vph(z)<r$
and $\log|z-\zeta|<r/\beta$, then (\ref{eq:holder}) implies
$\vph(\zeta)\le r+\log2$. Therefore, the second inequality in
(\ref{eq:Dap}) with $\log\delta=r/\beta$ yields
$$\Lambda(\Df_ku,\vph,r)\le \Lambda(u,\vph,r +\log2) -\frac{n}{k\beta}\,r +\frac{C_2}{k},$$
where $\Lambda$ is defined by (\ref{eq:Lambda}),
and then (\ref{eq:reltype}) gives us the first inequality in (\ref{eq:d2}).
\end{proof}

\medskip
This implies the convergence $\sigma(\Df_ku,\psi)\to\sigma(u,\psi)$ for $\psi\in \AAW$; moreover, it turns out to be one more characteristic property of this class of weights.

\begin{theorem}\label{cor:converg} A weight $\psi\in \MW$ is asymptotically  analytic if and only if \beq\label{eq:convtypes}\sigma(\Df_ku,\psi)\underset{k\to\infty}\longrightarrow\sigma(u,\psi)\quad \forall u\in\PSH_0.\eeq
\end{theorem}

{\it Remark.} This finishes the proof of Theorem A.

\begin{proof} Relation (\ref{eq:convtypes}) for $\psi\in\AAW$ follows from (\ref{eq:ass}),  Proposition~\ref{cor:convergeps} and Lemma~\ref{prop:appr} because all the functions $\exp\psi_\epsilon$ are H\"older continuous (say, with exponents $\beta_\epsilon$):
$$ \frac1{1+\epsilon}\left({\sigma(u,\psi_\epsilon)}-\frac{n}{k\beta_\epsilon}\right)\le \frac{\sigma(\Df_ku,\psi_\epsilon)}{1+\epsilon} \le\sigma(\Df_ku,\psi)\le \frac{\sigma(\Df_ku,\psi_\epsilon)}{1-\epsilon}\le \frac{\sigma(u,\psi_\epsilon)}{1-\epsilon}.
$$
Conversely, relation (\ref{eq:convtypes}) implies $\sigma(\Df_k\psi,\psi)\to1$ and, since $\sigma(\psi,\Df_k\psi)\ge 1$, the inequalities $\psi+O(1)\le \Df_k\psi\le (1-\epsilon)\psi+O(1)$ for $k\ge k_\epsilon$.
\end{proof}

\subsection{Representation by divisorial valuations}

Here we describe inf-analytic weights in terms of envelopes of divisorial psh valuations. This is a variant of the corresponding results from \cite{BFaJ}, presented from the pluripotential theory point of view.

Let $\mu$ be a proper modification over a neighborhood $X$ of $0$. A {\it divisorial valuation} on functions $f\in\cO_0$ is the generic multiplicity of $\mu^*f$ over an irreducible component of the exceptional divisor $\mu^{-1}(0)$. One extends this notion to psh functions $u$ by replacing the multiplicity of $\mu^*f$ with the generic Lelong number (\ref{eq:gLel}) of $\mu^*u$ over the component.
Here we will represent types relative to inf-analytic weights as envelopes of such valuations.

Let us take an analytic weight $\varphi=\log|F|\in \Wx$. By the Hironaka desingularization theorem, there exists a log resolution for the mapping $F$, i.e., a proper holomorphic mapping $\mu$ of a manifold $\hat X$ to a neighborhood $U$ of $0$, that is an isomorphism between $\hat X\setminus\mu^{-1}(0)$ and $U\setminus\{0\}$, such that $\mu^{-1}(0)$ is a normal crossing divisor with components $E_1,\ldots,E_N$, and in local coordinates centered at a generic point $p$ of a nonempty intersection $ E_I=\cap_{i\in I} E_{i}$, where $I\subset \{1,\ldots,N\}$,
the pullback of $F$ has the form
$$(F\circ\mu)(\hat x)=h(\hat x)\prod_{i\in I}\hat x_i^{m_i}$$
with $h(0)\neq 0$. Then for any $u\in \PSH_0$, the type relative to the weight $\varphi=\log|F|$ computes evidently as
\beq\label{eq:minnu}\sigma(u,\varphi)=\min \{\nu_{I,m_I}(\mu^*u): E_I\neq\emptyset\},\eeq where
$$ \nu_{I,m_I}(\mu^*u)=
\liminf_{\hat x\to 0}\frac{(\mu^*u)(\hat x)}{\sum_{i\in I} m_i\log|\hat x_i|}$$
at a generic point $p\in E_I$.

We need the following elementary result.

\begin{lemma}\label{lem:rees} Let $v(t)$ be a negative convex function on $\Rnk$, increasing in each $t_i$. Then
\beq\label{eq:elemlemma}\liminf_{t\to-\infty}\frac{v(t)}{\sum_it_i}= \min_i\liminf_{t_i\to-\infty}\frac{v(t)}{t_i}.\eeq
\end{lemma}

\noindent
\begin{proof} Denote the left hand side of (\ref{eq:elemlemma}) by $A$ and
$$A_i=\liminf_{t_i\to-\infty}\frac{v(t)}{t_i}.$$
From the convexity of $v$ it follows that for any point $t^*$, the ratio
$$\frac{v(t_1,t_2^*,\ldots,t_k^*)-v(t_1^*,t_2^*,\ldots,t_k^*)}{t_1-t_1^*}$$
decreases when $t_1\to-\infty$ and thus
$$ v(t_1,t_2^*,\ldots,t_k^*)-v(t_1^*,t_2^*,\ldots,t_k^*)\le A_1(t_1-t_1^*),\quad t_1<t_1^*.$$
Similarly,
$$ v(t_1,t_2,t_3^*,\ldots,t_k^*)-v(t_1,t_2^*,t_3^*,\ldots,t_k^*)\le A_2(t_2-t_2^*),\quad t_2<t_2^*,$$
and so on, the last inequality being
$$ v(t_1,\ldots,t_k)-v(t_1,\ldots,t_{k-1},t_k^*)\le A_k(t_k-t_k^*),\quad t_k<t_k^*.$$
Summing up the inequalities, we get
$$v(t)-v(t^*)\le \sum_iA_i(t_i-t_i^*)\le\min_iA_i\left(\sum_i t_i-\sum_i t_i^*\right),$$
which gives $A\ge\min_iA_i$. The reverse inequality is evident.
\end{proof}

\medskip

From Lemma~\ref{lem:rees} applied to the function $$v(t)=\sup\{\mu^*u(\hat x):\: \log|\hat x_i|<m_i^{-1}t_i,\ i\in I\}$$ we deduce
$$\nu_{I,m_I}(\mu^*u)= \min_{i\in I}\,\nu_{i,m_i}(\mu^*u)=\min_{i\in I}\,m_i^{-1}\nu_{E_i}(\mu^*u),$$
where $\nu_{E_i}(\mu^*u)$ is the generic Lelong number (\ref{eq:gLel}) of $\mu^*u$ along $E_i$. Following \cite{BFaJ}, we will call it the {\it divisorial valuation} of $u\in\PSH_0$ along $E_i$:
$$ {\cal R}_{\mu,i}(u)=\nu_{E_i}(\mu^*u)=\inf\{\nu(\mu^*u,p):\:p\in E_i\}.
$$

Now (\ref{eq:minnu}) gives us the following result.

\begin{theorem}\label{theo:anue} For any weight $\vph=\log|F|$, there exist
finitely many divisorial valuations ${\cal R}_j$ and positive integers $m_j$ such that
$\sigma(u,\vph)=\min_j\, m_j^{-1}{\cal R}_j(u)$ for every $u\in\PSH_0$.
\end{theorem}

{\it Remark.} For the case $u=\log|f|$, this follows also from \cite[Thm. 4.1.6]{LT} because the Samuel asymptotic function $\bar\nu_\cI(f)$ (\ref{eq:Sam}) with respect to the primary ideal $\cI( F_1,\ldots,F_N)$  coincides with the relative type of $\log|f|$ with respect to the weight $\log|F|$. Conversely, one can deduce Theorem~\ref{theo:anue} from that result by applying Theorem~\ref{cor:converg}.

\begin{theorem}\label{theo:anue1}  If $\psi\in\MW$, then $\psi\in \UAW$ if and only if there exist denumerably many
divisorial valuations ${\cR}_j$ and positive numbers $s_j$ such that
\beq\label{eq:inf}\sigma(u,\psi)=\inf_j\, s_j\,{\cR}_j(u)\quad\forall u\in\PSH_0.\eeq
\end{theorem}

{\it Remark.} This completes the proof of Theorem B.

\begin{proof} Let $\psi\in\UAW$. By Theorem~\ref{theo:anue}, the type $\sigma(u,\Df_{m!}\,\psi)$ for each $m$ is the lower envelope of finitely many weighted divisorial valuations $s_j{\cR}_j(u)$. Then (\ref{eq:inf}) follows from Proposition~\ref{cor:cont2}.

Conversely, (\ref{eq:inf}) means that $G_\psi$ is the best psh minorant of the function $\inf_j G_{\psi_j}$, where the weights $\psi_j$ are such that $\sigma(u,\psi_j)=s_j{\cal R}_j(u)$, see the proof of Proposition~\ref{cor:cont2}. All $\psi_j$ are, by \cite[Thm. 2.13]{BFaJ}, tame (and thus asymptotically analytic, see Example~\ref{ex:tame}). Since the best psh minorant of the minimum of finitely many tame weights is tame as well,
there exists a sequence of weights $\vph_j$ with analytic singularities whose Green functions $G_{\vph_j}$ decrease to $G_\psi$. Therefore, $\psi\in\UAW$.
\end{proof}
\medskip

{\it Remarks}. 1. For tame weights, the representation (\ref{eq:inf}) is proved in \cite{BFaJ}.
For asymptotically analytic weights, the arguments were sketched in \cite{R9}.

2. Relative types have the obvious property $\sigma(\sum \alpha_j\, u_j,\vph)\ge \sum\alpha_j\, \sigma(u_j,\vph)$ and hence are concave functionals on $\PSH_0$, while for the divisorial valuations $\cR_j$ the inequality becomes an equality. From this point of view, relation (\ref{eq:inf}) is similar to the representation of concave functions as lower envelope of linear ones, which holds on linear topological spaces.
This can be put into the picture of tropical analysis on psh singularities \cite{R7}, \cite{R9}.

\subsection{Analytic disks}
Another (although related) representation for the relative types can be given by means of analytic disks. To do so, we use arguments from \cite{LT}.
Given $u\in\PSH_0$ and $\vph\in \MW$, denote
\beq\label{eq:curve}
\sigma^*(u,\vph)=\inf_{\gamma\in \cA}\liminf_{\zeta\to 0}\frac{\gamma^*u(\zeta)}{\gamma^*\vph(\zeta)},\eeq
where $\cA$ is the collection of all analytic maps $\gamma:\D\to X$ such that $\gamma(0)=0$, and $\gamma^*u$ is the pullback of $u$ by $\gamma$.

Evidently, $\sigma^*(u,\vph)\ge\sigma(u,\vph)$ for any $\vph\in \MW$.
 If $f$ is a holomorphic function and $F$ a holomorphic mapping with isolated zero at $0$, then $\sigma^*(\log|f|,\log|F|)=\sigma(\log|f|,\log|F|)$ by \cite[Thm. 5.2]{LT}; moreover, in this case there exists a curve  $\gamma\in\cA$ such that
$$\sigma(\log|f|,\log|F|)=\liminf_{\zeta\to 0}\frac{\log|\gamma^*f(\zeta)|}{\log|\gamma^*F(\zeta)|},$$
see \cite[Prop. 5.4]{LT}. Similar arguments together with
Theorems~\ref{theo:anue} and \ref{theo:anue1} give us the following result.

\begin{theorem}\label{theo:ancurve} If $\psi\in\UAW$, then $\sigma^*(u,\psi)=\sigma(u,\psi)$
for every $u\in\PSH_0$. If, in addition, $\psi$ has analytic singularity, then the infimum in (\ref{eq:curve}) always attains.
\end{theorem}

\begin{proof} For the case of analytic weight $\psi=\log|F|$ we will use the idea from the proof of \cite[Prop. 5.4]{LT}.
By Theorem~\ref{theo:anue}, $\sigma(u,\log|F|)=\min_j\, m_j^{-1}{\cal R}_j(u)$, where $m_j$ and ${\cal R}_j(u)$ are generic Lelong numbers of $\log|\mu^*F|$ and $\mu^*u$, respectively, along the exceptional primes $E_j$ of a log resolution $\mu$ for the mapping $F$.  Take a germ of an analytic curve $\hat\gamma_j\subset\hat X$ passing transversally through a point $\hat x_j\in E_j$ and such that the generic Lelong number of $\mu^*u$ along $E_i$ equals the Lelong number of the restriction of $\mu^*u$ to $\hat\gamma_j$ at $\hat x_j$ (which is possible by Siu's theorem). Then, for the curve $\gamma_j=\mu^*\hat\gamma_j$, we get the equalities
$$ {\cal R}_j(u)=\liminf_{\zeta\to 0}\frac{\gamma_j^*u\,(\zeta)}{\log|\gamma_j(\zeta)|},\quad
m_j=\lim_{\zeta\to 0}\frac{\log|\gamma_j^*F(\zeta)|}{\log|\gamma_j(\zeta)|},
$$
and so, $$\sigma(u,\log|F|)=\min_{j}\liminf_{\zeta\to 0}\frac{\gamma_j^*u\,(\zeta)}{\log|\gamma_j^*F(\zeta)|}.$$

For arbitrary $\psi\in\UAW$, we refer then to Theorem~\ref{theo:anue1}.
\end{proof}

\bigskip

{\small

{\bf Acknowledgment.} The author thanks the anonymous referee for valuable suggestions that have helped to improve the presentation.

}


\begin{thebibliography}{11}

\bibitem{ACCP}
{\sc P. {\AA}hag, U. Cegrell, R. Czyz, Ph\d{a}m Ho\`{a}ng Hi\d{\^{e}}p},
{\it Monge-Amp\`ere measures on pluripolar sets}, J. Math. Pures Appl. (9) {\bf 92} (2009), no. 6, 613--627.

\bibitem{Bl}
{\sc Z. Blocki}, {\it Some applications of the Ohsawa-Takegoshi extension theorem}, Expo. Math. {\bf 27} (2009), no. 2, 125–-135.

\bibitem{BFaJ}
{\sc  S. Boucksom, C. Favre, and M. Jonsson}, {\it Valuations and
plurisubharmonic singularities}, Publ. Res. Inst. Math. Sci. {\bf 44} (2008), no. 2, 449--494.

\bibitem{C2}
{\sc U. Cegrell}, {\it The general definition of the complex
Monge--Amp\`ere operator}, Ann. Inst. Fourier (Grenoble) {\bf 54}
(2004), no. 1, 159--179.



\bibitem{D2}
{\sc J.-P. Demailly}, {\it Regularization of closed positive currents
  and intersection theory}, J. Algebraic Geometry {\bf 1} (1992), 361--409.

\bibitem{D1}
{\sc J.P. Demailly}, {\it  Monge-Amp\`ere operators, Lelong numbers and
intersection theory,}
Complex Analysis and Geometry (Univ. Series in Math.), ed. by V. Ancona
and A. Silva, Plenum Press, New York 1993, 115--193.

\bibitem{D6}
{\sc J.-P. Demailly}, {\it Multiplier ideal sheaves and analytic methods in algebraic geometry}, School on Vanishing Theorems and Effective Results in Algebraic Geometry (Trieste, 2000), 1--148, ICTP Lect. Notes, 6, Abdus Salam Int. Cent. Theoret. Phys., Trieste, 2001.

\bibitem{D8}
{\sc J.-P. Demailly}, {\it Estimates on Monge--Amp\`ere operators derived from a local algebra inequality}, in: Complex Analysis and Digital Geometry. Proceedings from the Kiselmanfest, 2006, Uppsala University, 2009, 131--143.

\bibitem{DEL}
{\sc J.-P. Demailly, L. Ein, and R. Lazarsfeld}, {\it Subadditivity property of multiplier ideals}, Michigan Math. J. {\bf 48} (2000), 137--156.


\bibitem{ELS1}
{\sc L. Ein, R. Lazarsfeld, and K.E. Smith}, {\it Uniform bounds and symbolic powers on smooth varieties}, Invent. Math. {\bf 144} (2001), no. 2, 241--252.

 \bibitem{ELS2}
{\sc L. Ein, R. Lazarsfeld, and K.E. Smith}, {\it Uniform approximation of Abhyankar valuation ideals in smooth function fields}, Amer. J. Math. {\bf 125} (2003), no. 2, 409--440.


\bibitem{FaG}
{\sc  C. Favre and V. Guedj}, {\it Dynamique des applications rationnelles des espaces multiprojectifs}, Indiana Univ. Math. J. {\bf 50} (2001), no. 2, 881--934.

\bibitem{FaJ}
{\sc  C. Favre and M. Jonsson}, {\it Valuative analysis of planar
plurisubharmonic functions}, Invent. Math. {\bf 162} (2005),
271--311.

\bibitem{FaJ2}
{\sc  C. Favre and M. Jonsson}, {\it Valuations and multiplier ideals}, J. Amer. Math. Soc.{\bf 18} (2005), no.~3,
655--684.

\bibitem{JM}
{\sc  M. Jonsson and M. Mus\c{t}at\v{a}}, {\it Valuations and asymptotic invariants for sequences of ideals},  to appear in Ann. Inst. Fourier (Grenoble); available at http://arxiv.org/abs/1011.3699

\bibitem{JM1}
{\sc  M. Jonsson and M. Mus\c{t}at\v{a}}, {\it An algebraic approach to the openness conjecture of Demailly and Koll\'{a}r}, available at http://arxiv.org/abs/1205.4273

\bibitem{Kl}
{\sc M. Klimek}, { Pluripotential theory}. Oxford University Press,
London, 1991.

 \bibitem{La}
{\sc R. Lazarsfeld}, Positivity in Algebraic Geometry, II. Springer-Verlag, Berlin, 2004.

\bibitem{LT}
{\sc M. Lejeune-Jalabert et B. Teissier},
 {\it Cl\^{o}ture int\'{e}grale des id\'eaux et \'equisingularit\'e},
Ann. Fac. Sci. Toulouse Math. (6) {\bf 17} (2008), no. 4, 781–859.


\bibitem{LeR}
{\sc P. Lelong and A. Rashkovskii}, {\it Local indicators for
plurisubharmonic functions}, J. Math. Pures Appl. {\bf 78} (1999),
233--247.

\bibitem{Mustata}
{\sc M. Mus\c{t}at\v{a}}, {\it On multiplicities of graded sequences of ideals},  J. Algebra {\bf 256} (2002), 229--249.

\bibitem{Na}
{\sc A.M. Nadel}, {\it Multiplier ideal sheaves and K\"{a}hler-Einstein mestrics of positive scalar curvature}, Annals Math. {\bf 132} (1990), 549--596.

\bibitem{NP}
{\sc  Nguyen Van Khue and Pham Hoang Hiep}, {\it A comparison principle for the complex Monge--Ampère operator in Cegrell's classes and applications}, Trans. Amer. Math. Soc. {\bf 361} (2009), no. 10, 5539--5554.

\bibitem{NR}
{D.G Northcott and D. Rees}, Reductions of ideals in local rings,
{\it Proc. Cambridge Philos. Soc.} {\bf 50} (1954), 145--158.

\bibitem{R3}
{\sc A. Rashkovskii}, {\it Newton numbers and residual measures of
  plurisubharmonic functions}, Ann. Polon. Math. {\bf 75} (2000),
no. 3, 213--231.

\bibitem{R-surv}
{\sc A. Rashkovskii}, {\it Singularities of plurisubharmonic functions and positive closed currents}, preprint at
http://arxiv.org/abs/math/0108159

\bibitem{R7}
{\sc A. Rashkovskii}, {\it Relative types and extremal problems for
plurisubharmonic functions}, Int. Math. Res. Not. {\bf 2006} (2006),
Article ID 76283, 26 p.

\bibitem{R9}
{\sc A. Rashkovskii}, {\it Tropical analysis of plurisubharmonic singularities}, Tropical and Idempotent Mathematics, 305--315, Contemp. Math., {\bf 495}, Amer. Math. Soc., Providence, RI, 2009.

\bibitem{RSig2}
{\sc A. Rashkovskii and R. Sigurdsson}, {\it Green functions with
singularities along complex spaces}, Internat. J. Math. {\bf 16}
(2005), no. 4, 333--355.

\bibitem{Sa}
{\sc P. Samuel}, {\it Some asymptotic properties of powers of ideals},
Ann. of Math. (2) {\bf 56} (1952). 11--21.


\bibitem{X}
{\sc Yang Xing}, {\it Weak convergence of currents},
Math. Z. {\bf 260} (2008), no. 2, 253--264.

\bibitem{Za0}
{\sc V.P. Zahariuta}, {\it Spaces of analytic functions and maximal
plurisubharmonic functions.} D.Sci. Dissertation, Rostov-on-Don,
1984.

\bibitem{Za} {\sc V.P. Zahariuta}, {\it Spaces of analytic functions and
Complex Potential Theory}, Linear Topological Spaces and Complex
Analysis {\bf 1} (1994), 74--146.



\end{thebibliography}
\end{document}